\documentclass[12pt,reqno]{siugrad51new}
\usepackage{amsmath, amssymb, amsthm, latexsym, mathtools,bold-extra}
\usepackage{adjustbox}
\makeatletter
\providecommand*{\dashv}{%
  \mathrel{%
    \mathpalette\@dashv\vdash
  }%
}
\newcommand*{\@dashv}[2]{%
  \reflectbox{$\m@th#1#2$}%
}
\usepackage{tikz-cd}
\usepackage[colorlinks=false,hyperfootnotes=false]{hyperref}
\renewcommand*\thesection{\arabic{section}}
\theoremstyle{plain}
\newtheorem{theorem}{Theorem}[section]
\newtheorem{mainthm}[theorem]{Main Theorem}
\newtheorem{corollary}[theorem]{Corollary}

\newtheorem{proposition}[theorem]{Proposition}
\theoremstyle{definition}
\newtheorem{example}[theorem]{Example}
\newtheorem{fact}[theorem]{Fact}
\newtheorem{definition}[theorem]{Definition}
\newtheorem{remark}[theorem]{Remark}

\newtheorem{ques}[theorem]{Question}
\newtheorem{obs}[theorem]{Observation}
\newtheorem{conj}[theorem]{Conjecture}

\newcommand{\lk}{\leq_{\bf K}}

\newcommand{\oop}{\operatorname}

\newcommand{\gs}{\oop{gS}}
\newcommand{\cf}{\oop{cf}}

\newcommand{\ran}{\oop{ran}}

\newcommand{\ls}{\oop{LS}({\bf K})}

\newcommand{\llk}{\oop{L}({\bf K})}
\newcommand{\defeq}{\vcentcolon=}

\def\fork{\mathrel{\raise0.2ex\hbox{\ooalign{\hidewidth$\vert$\hidewidth\cr\raise-0.9ex\hbox{$\smile$}}}}}

\newcommand{\nr}[1]{\lVert #1 \rVert}

\newcommand{\al}{{\aleph_0}}

\makeatletter
\@ifpackageloaded{hyperref}%
  {\newcommand{\mylabel}[2]
    {\protected@write\@auxout{}{\string\newlabel{#1}{{#2}{\thepage}%
      {\@currentlabelname}{\@currentHref}{}}}}}%
  {\newcommand{\mylabel}[2]
    {\protected@write\@auxout{}{\string\newlabel{#1}{{#2}{\thepage}}}}}
\makeatother
\begin{document}

\pagenumbering{roman}
\setcounter{page}{0}
\newpage
\pagenumbering{arabic}
\setcounter{page}{1}
\parindent=.35in
\begin{center}
        \begin{center}%
         {\Large\bfseries\scshape Axiomatizing AECs and Applications}\\\vspace{1em}{\scshape Samson Leung}\\      
        \end{center}%
\end{center}
{\let\thefootnote\relax\footnote{Date: \today
}} 
\begin{abstract}
For any abstract elementary class (AEC) ${\bf K}$ with $\lambda=\ls$, the following holds:
\begin{enumerate}
\item $K$ has an axiomatization in $L_{(2^\lambda)^+,\lambda^+}$, allowing game quantification. If ${\bf K}$ has arbitrarily large models, the $\lambda$-amalgamation property and is categorical both in $\lambda$ and $\lambda^+$, then it has an axiomatization in $L_{\lambda^{+},\lambda^{+}}$ with game quantification. These extend Kueker's \cite{kue2} result which assumes finite character and $\lambda=\al$.
\item If $K$ is universal and categorical in $\lambda$, then it is axiomatizable in $L_{\lambda^+,\lambda^+}$.
\item Shelah's celebrated presentation theorem asserts that for any AEC ${\bf K}$ there is a first-order theory in an expansion of $\llk$, and a set $\Gamma$ of $2^\lambda$ many $T$-types such that $K=PC(T,\Gamma,\llk)$. We provide a better bound on $|\Gamma|$ in terms of $I_2(\lambda,{\bf K})$.
\item We present additional applications which extend, simplify and generalize results of Shelah \cite{sh88,sh576} and Shelah-Vasey \cite{ss}. Some of our main results generalize to $\mu$-AECs.
\end{enumerate}

\end{abstract}
\vspace{1em}
\begin{center}
{\bfseries TABLE OF CONTENTS}
\end{center}
\tableofcontents

\section{Introduction}
In the proof of Shelah's presentation theorem \cite[I Lemma 1.9]{shh}, functions are added to capture isomorphism axioms and L\"{o}wenheim-Skolem axiom. \cite[Theorem 2.1]{sv} claimed that any abstract elementary class (AEC) ${\bf K}$ can be axiomatized by an $L_{\beth_{2}(\lambda)^{+3},\lambda^+}$ sentence where $\lambda$ is the L\"{o}wenheim-Skolem number, and such an axiomatization is in $\llk$.

Let $\chi\defeq\lambda+I_2(\lambda,{\bf K})$, where $I_2(\lambda,{\bf K})$ is the number of nonisomorphic pairs $(M,N)$ such that $M,N\in K_\lambda$ and $M\lk N$. In \ref{charmod}, we will axiomatize an AEC ${\bf K}$ by a sentence $\sigma_{\bf K}$ in $L_{\chi^+,\lambda^+}$, allowing game quantification. As $\chi\leq2^\lambda$, we have that $\sigma_{\bf K}$ is in $L_{(2^\lambda)^+,\lambda^+}$, improving Shelah and Villaveces' result. Modulo the use of game quantification, our result is optimal for uncountable $\lambda$ as it is known that there is an AEC that cannot be axiomatized by an $L_{\infty,\lambda}$ sentence \cite{hsimon}. Under extra assumptions, we can axiomatize an AEC by a sentence in  $L_{\lambda^{+},\lambda^{+}}$ (\ref{charuni} and \ref{charcat2}). By slightly modifying $\sigma_{\bf K}$, we can encode the ${\bf K}$-substructure relation (\ref{charprop} and \ref{charsub}) by a formula $\sigma_{\leq}$ in $L_{\chi^+,\lambda^+}$. As above, we can improve the results under extra assumptions (\ref{charpropcor}).

As an application of our axiomatization of AECs, we derive a variation on the presentation theorem, where any AEC is a $PC_\chi$ class (\ref{presvar}) and game quantification is not used. Our presentation theorem is stronger than Shelah's as the bound of $|\Gamma|$ in some cases is smaller than $2^{\ls}$. It also lowers the threshold of the existence theorem from successive categoricity (\ref{presvarsuc}). The axiomatization strategy is also applicable to the $\mu$-AEC analogs, giving a stronger presentation theorem (\ref{presvarmu}) than \cite[Theorem 3.2]{muaec}.

In the following, we provide two tables. The first table summarizes the known results in literature. The definition of $L(\omega)$ can be found in \cite[Definition 1.12]{kue2} and $L_{\chi^+,\lambda^+}(\omega\cdot\omega)$ is defined in \ref{gamedef}. The second table summarizes the main results in this paper. We write $AL=$ arbitrarily large models, $AP=$ amalgamation property, $JEP=$ joint embedding property, $NMM=$ no maximal models. Monster model means $AP+JEP+NMM$. The entries of the second table link to the related theorems (see \ref{charobs} for cases where we do not assume $\lambda$-categoricity).

\pagebreak

\begin{table}[t!]\centering
\begin{center}\underline{Known results}\end{center}
\begin{tabular}{|l|l|l|}
\hline
Assumptions on ${\bf K}$ &${K}$ is&References\\ \hline
None&axiomatizable in $L_{(2^{2^{\lambda^+}})^{+++},\lambda^+}$&\cite[Theorem 2.1]{sv}\\ \hline 
None&$PC_{\lambda,2^\lambda}$&\cite[Lemma 1.8]{sh88}\\ \hline
None& reducts of a theory in $L'_{\chi^+,\lambda^+}$ &\cite[Theorem 3.2.3]{bbhanf}\\ 
&where $L'\supseteq L$&\\ \hline
$\lambda=\aleph_0$, $\al$-$AP$, stable in $\al$,&$PC_{\al}$&\cite[Theorem 4.2]{ss}\\ 
$I(\al,{\bf K})\leq\al$&&\\\hline
$\lambda=\al$&closed under $\equiv_{\infty,\omega_1}$&\cite[Theorem 2.5]{kue2}\\ \hline
$\lambda=\al$, $\exists\kappa=\kappa^{\al}(I(\kappa,{\bf K})\leq\kappa)$&axiomatizable in $L_{\infty,\omega_1}$&\cite[Theorem 2.11]{kue2}\\ \hline
$\lambda=\al$, $\exists\kappa=\kappa^{\al}(I(<\kappa,{\bf K})\leq\kappa)$&axiomatizable in $L_{\kappa^+,\omega_1}$&\cite[Theorem 2.11]{kue2}\\ \hline
Finitary, $\lambda=\al$&closed under $\equiv_{\infty,\omega}$&\cite[Theorem 3.4]{kue2}\\ \hline
Finitary&closed under $\equiv_{\infty,\omega}$&\cite[Theorem 3.7]{gj}\\ \hline
Finitary, $\lambda=\al$&axiomatizable in $L(\omega)$&\cite[Theorem 3.7]{kue2}\\ \hline
Finitary, $\lambda=\al$, $\exists\kappa(I(\kappa,{\bf K})\leq\kappa)$&axiomatizable in $L_{\infty,\omega}$&\cite[Theorem 3.10]{kue2}\\ \hline
Finitary, $\exists\kappa=\kappa^{<\lambda}(I(\kappa,{\bf K})\leq\kappa)$&axiomatizable in $L_{\infty,\lambda}$&\cite[Theorem 3.10]{gj}\\ \hline
Finitary, $\lambda=\al$,&axiomatizable in $L_{\kappa^+,\omega}$&\cite[Theorem 3.10]{kue2}\\ 
$\exists\kappa(I(<\kappa,{\bf K})\leq\kappa)$&&\\\hline
Finitary, $\exists\kappa=\kappa^{<\lambda}(I(<\kappa,{\bf K})\leq\kappa)$&axiomatizable in $L_{\kappa^+,\lambda}$&\cite[Theorem 3.10]{gj}\\ \hline
$\lambda=\al$, $\exists\kappa(I(\kappa,{\bf K})=1)$&$K_{\geq\kappa}$ is closed under $\equiv_{\infty,\omega}$&\cite[Theorem 5.1]{kue2}\\ \hline
Finitary, $\lambda=\al$, $\exists\kappa(I(\kappa,{\bf K})=1)$&$K_{\geq\kappa}$ axiomatizable in $L_{\omega_1,\omega}$&\cite[Theorem 5.2]{kue2}\\ \hline
$\lambda=\al$, $\exists\kappa=\kappa^{\al}(I(\kappa,{\bf K})=1)$&Models in $K_{\geq\kappa}$ are $\equiv_{\infty,\omega_1}$&\cite[Theorem 5.3a]{kue2}\\ \hline
Previous row $+$ monster model&$K_{\geq\kappa}$ axiomatizable in $L_{(2^\omega)^+,\omega_1}$&\cite[Theorem 5.3c]{kue2}\\ \hline
$\lambda>\al$ &$K$ is closed under $\equiv_{\infty,\lambda^+}$&\cite[Theorem 7.2]{kue2}\\ \hline
Monster model, $\lambda>\al$, &$K_{\geq\kappa}$ axiomatizable in $L_{{2^\kappa}^+,\kappa^+}$&\cite[Theorem 7.4]{kue2}\\
$\exists\kappa(I(\kappa,{\bf K})=1\wedge\cf(\kappa)>\lambda)$&&\\\hline
Finitary, monster model, &$K_{\geq\kappa}$ axiomatizable in $L_{\infty,\lambda}$&\cite[Theorem 3.11]{gj}\\
$\exists\kappa(I(\kappa,{\bf K})=1\wedge\cf(\kappa)>\lambda)$&&\\\hline
\end{tabular}
\end{table}
\pagebreak

\begin{table}[t!]\centering
\begin{center}\underline{New results}\end{center}
\begin{tabular}{|l|l|l|}
\hline
Assumptions on ${\bf K}$ & Axiomatization in & $K$ is \\ \hline
None & \hyperref[charmod]{$L_{\chi^+,\lambda^+}(\omega\cdot\omega)$}&\hyperref[presvar]{$PC_{\chi}$}\\ \hline
Universal class, $I(\lambda,{\bf K})\leq\lambda$ & \hyperref[charuni]{$L_{\lambda^+,\lambda^+}$} & \hyperref[presvarcot]{$PC_\lambda$} \\ \hline
$AL$, $\lambda$-$AP$, $I(\lambda,{\bf K})\leq\lambda$, $I(\lambda^+,{\bf K})=1$ &   \hyperref[charcat2]{$L_{\lambda^{+},\lambda^{+}}(\omega\cdot\omega)$}&\hyperref[presvarcot]{$PC_{\lambda}$}\\ \hline
$2^\lambda<2^{\lambda^+}$, $AL$, $I(\lambda,{\bf K})=1$, $I(\lambda^+,{\bf K})=1$  &   \hyperref[charcat2cor]{$L_{\lambda^{+},\lambda^{+}}(\omega\cdot\omega)$}&\hyperref[presvarcot]{$PC_{\lambda}$}\\\hline
$\lambda$-$AP$, $I(\lambda,{\bf K})\leq\lambda$, $I(\lambda^+,{\bf K})=1$&   \hyperref[charcat2]{$L_{\lambda^{+},\lambda^{+}}(\omega\cdot\omega)$}&\hyperref[presvarcot]{$PC_{\lambda}$}\\ 
stable in $\lambda$  &&\\\hline
$2^\lambda<2^{\lambda^+}$, $I(\lambda,{\bf K})=1$, $I(\lambda^+,{\bf K})=1$, &   \hyperref[charcat2cor]{$L_{\lambda^{+},\lambda^{+}}(\omega\cdot\omega)$}&\hyperref[presvarcot]{$PC_{\lambda}$}\\
stable in $\lambda$&&\\\hline
\end{tabular}
\end{table}

\pagebreak

Note that the last row is a significant improvement of \cite[Theorem 4.2]{ss}, using much simpler and general methods while covering the case when $\lambda$ is uncountable.
We highlight the differences between our result and \cite[Theorem 3.2.3]{bbhanf}:
\begin{enumerate}
\item They expand the base vocabulary to $\tau^*$ by adding new predicates of arity $\lambda$, and their theory $T^*$ in the expanded language is more semantic and longer; our axiomatization keeps the original language $\llk$ and is purely syntactic. 
\item Their relational presentation theorem characterizes $K$ as reducts of models of $T^*$, ${\bf K}$-substructure as reducts of $\tau^*$-substructure; our axiomatization is simply in $L_{\chi^+,\lambda^+}(\omega\cdot\omega)$ and we pin down the formula that determines ${\bf K}$-substructure. 
\item Their expanded language $\tau^*$ has size $\chi$; our axiomatization uses the original language so has size $\leq\lambda$ (but both approaches require taking $\chi$-conjunctions). We expand the language to size $\chi$ only when we derive a variation on Shelah's presentation theorem. 
\item In \cite[Theorem 3.2.3]{bbhanf}, they do not require types to be omitted because their theory $T^*$ is in the infinitary logic. We omit types in our variation to Shelah's presentation theorem so as to represent $K$ as a first-order $PC$ class.
\end{enumerate}

Our approach in this paper was inspired by Villaveces' question of the complexity of the example in \cite[Proposition 4.1]{leung1}, which has high instability but low complexity $\gamma=1$. Also, Grossberg suggested in May 2021 that \cite{sv} could have a significant improvement. This motivated us to look for a simpler axiomatization of an AEC, without using trees or other combinatorial machinery in \cite[Theorem 2.4]{sv}. At the cost of game quantification, we lower the complexity of junctions in their paper from $\beth_{2}(\lambda)^{+3}$ to $\chi^+$.

This paper was written while the author was working on a Ph.D. under the direction of Rami Grossberg at Carnegie Mellon University and we would like to thank Prof. Grossberg for his guidance and assistance in my research in general and in this work in particular.

\section{Preliminaries}
Let $L$ be a finitary language, $\lambda_1\geq\lambda_2$ be infinite cardinals. We write $L_{\lambda_1,\lambda_2}$ the set of formulas generated by $L$, allowing $<\lambda_2$ free variables and $<\lambda_2$ quantifiers, in addition to conjunctions and disjunctions of $<\lambda_1$ subformulas. Given an $L$-structure $M$, we write $|M|$ the universe of $M$ and $\nr{M}$ the cardinality of $M$.

\begin{definition}
Let $L\subseteq L'$ be two langauges (they can be infinitary), $T$ be an $L'$-theory and $\Gamma$ be a set of $L'$-types. Let $\mu$ be a regular cardinal. If $L,L'$ are $(<\mu)$-ary, we define 
\begin{align*}
EC^\mu(T,\Gamma)&\defeq\{M:M\text{ is an $L'$-structure},M\vDash T\ ,M\text{ omits }\Gamma\}\\
PC^\mu(T,\Gamma)&\defeq\{M\restriction L:M\text{ is an $L'$-structure},M\vDash T\ ,M\text{ omits }\Gamma\}
\end{align*}
When $\mu=\al$, we omit the superscript $\al$. 

Let $\lambda,\chi$ be infinite cardinals, and assume $|T|\leq\lambda$ and $|\Gamma|\leq\chi$. If $K=EC^\mu(T,\Gamma)$, we call $K$ an $EC^\mu_{\lambda,\chi}$ class. If $K=PC^\mu(T,\Gamma,L)$, we call $K$ a $PC^\mu_{\lambda,\chi}$ class. We omit the superscript $\al$ when $\mu=\al$. We omit $\lambda,\chi$ if the sizes of $T$ and $\Gamma$ are not specified. $PC_\lambda$ stands for $PC_{\lambda,\lambda}$.
\end{definition}

\begin{definition}\mylabel{aecdef}{Definition \thetheorem}
Let $L$ be a finitary language. An abstract elementary class ${\bf K}=\langle K,\lk\rangle$ in $L$ satisfies the following axioms:
\begin{enumerate}
\item $K$ is a class of $L$-structures and $\lk$ is a partial order on $K\times K$.
\item For $M_1,M_2\in K$, $M_1\lk M_2$ implies $M_1\subseteq M_2$ (as $L$-substructures).
\item Isomorphism axioms:
\begin{enumerate}
\item If $M\in K$, $N$ is an $L$-structure, $M\cong N$, then $N\in K$.
\item Let $M_1,M_2,N_1,N_2\in K$. If $f:M_1\cong M_2$, $g:N_1\cong N_2$, $g\supseteq f$ and $M_1\lk N_1$, then $M_2\lk N_2$.
\end{enumerate}
\item Coherence: Let $M_1,M_2,M_3\in K$. If $M_1\lk M_3$, $M_2\lk M_3$ and $M_1\subseteq M_2$, then $M_1\lk M_2$.
\item L\"{o}wenheim-Skolem axiom: There exists an infinite cardinal $\lambda\geq|\llk|$ such that: for any $M\in K$, $A\subseteq |M|$, there is some $N\in K$ with $A\subseteq|N|$, $N\lk M$ and $\nr{N}\leq\lambda+|A|$. We call the minimum such $\lambda$ the L\"{o}wenheim-Skolem number $\ls$. 
\item Chain axioms: Let $\alpha$ be an ordinal and $\langle M_i:i<\alpha\rangle\subseteq K$ such that for $i<j<\alpha$, $M_i\lk M_j$.
\begin{enumerate}
\item Then $M\defeq\bigcup_{i<\alpha}M_i$ is in $K$ and for all $i<\alpha$, $M_i\lk M$.
\item Let $N\in K$. If in addition for all $i<\alpha$, $M_i\lk N$, then $M\lk N$.
\end{enumerate}
\end{enumerate}
Let $\lambda\geq\ls$ be a cardinal. We define $K_{\lambda}\defeq\{M\in K:\nr{M}=\lambda\}$ and ${\bf K_{\lambda}}\defeq\langle K_{\lambda},\lk\restriction K_{\lambda}\times K_{\lambda}\rangle$. When the context is clear, we omit the subscript of $\lk$ and write $\leq$. We will only consider ${\bf K_{\geq\ls}}$ in place of ${\bf K}$, which is still an AEC.
\end{definition}
\begin{definition}
\begin{enumerate}
\item Let $I$ be an index set. A directed system $\langle M_i:i\in I\rangle\subseteq K$ indexed by $I$ satisfies the following: for any $i,j\in I$, there is $k\in I$ such that $M_i\leq M_k$ and $M_j\leq M_k$.
\item Let $\mu$ be an infinite cardinal. A $\mu$-directed system $\langle M_i:i\in I\rangle\subseteq K$ indexed by $I$ satisfies the following: for any $J\subseteq I$ of size $<\mu$, there is $k\in I$ such that for all $j\in J$, $M_j\leq M_k$, (thus a system is directed iff it is $\al$-directed.)
\end{enumerate}
\end{definition}
\begin{fact}\cite[I Observation 1.16]{shh}\mylabel{dirfact}{Fact \thetheorem}\hfill
 Let $\langle M_i:i\in I\rangle\subseteq K$ be a directed system. Then
\begin{enumerate}
\item $M\defeq \bigcup_{i\in I}M_i\in K$
\item For all $i\in I$, $M_i\leq M$.
\item Let $N\in K$. If in addition for all $i\in I$, $M_i\leq N$, then $M\leq N$. 
\end{enumerate}
\end{fact}
\begin{fact}\cite[II Claim 1.8(2)]{shh}\mylabel{dirfact2}{Fact \thetheorem} If $M\leq N$ in $K$, then there are index sets $I_1$ and $I_2$, directed systems $\langle M_i:i\in I_1\rangle$ and $\langle N_i:i\in I_2\rangle$ of union $M$, $N$ respectively, $I_1\subseteq I_2$ and $M_i=N_i$ for all $i\in I_1$.
\end{fact}

\begin{definition}\mylabel{gamedef}{Definition \thetheorem}
Let $L$ be a language, $\lambda,\chi$ be infinite cardinals and $\delta$ be an ordinal. $L_{\lambda,\chi}(\delta)$ extends $L_{\lambda,\chi}$ by allowing $\delta$-game quantification: if $\phi$ is a formula in $L_{\lambda,\chi}(\delta)$ with free variables $(x_\alpha,y_\alpha)_{\alpha<\delta}$ and $l(x_\alpha),l(y_\alpha)<\chi$, then $(\forall x_\alpha\exists y_\alpha)_{\alpha<\delta}\,\phi $ is a formula in $L_{\lambda,\chi}(\delta)$. An $L$-structure $M$ satisfies $(\forall x_\alpha\exists y_\alpha)_{\alpha<\delta}\,\phi $ if Player II has a winning strategy in the following game of $\delta$ rounds: in the $\alpha$-th round, Player I chooses some tuple $a_\alpha\subseteq M$ of length $l(x_\alpha)$ and Player II responds by choosing some tuple $b_\alpha\subseteq M$ of length $l(y_\alpha)$. Player II wins if $M\vDash\phi[a_\alpha,b_\alpha]_{\alpha<\delta}$ where for $\alpha<\delta$, $x_\alpha$ is substituted by $a_\alpha$ and $y_\alpha$ is substituted by $b_\alpha$.
\end{definition}
If $\delta$ above is finite, then $L_{\lambda,\chi}(\delta)=L_{\lambda,\chi}$. The use of game quantifiers in AECs can be found in \cite[Theorems 2.9, 3.7]{kue2} which handle the case $\ls=\al$. Our version is consistent with $L(\omega)$ there and is called a \emph{closed game quantifier} in \cite[Chapter X.2]{ko}.
\section{Encoding an AEC}
In this section, we fix an AEC ${\bf K}$ in a language $L$. 
\begin{definition}\mylabel{chap2ques1}{Definition \thetheorem}
\begin{enumerate}
\item Let $\lambda\geq\ls$. $I(\lambda,{\bf K})\defeq\{M/_{\cong}:M\in K_\lambda\}$.
\item Let $M_1\leq N_1$, $M_2\leq N_2$. We write $(M_1,N_1)\cong(M_2,N_2)$ if there exists $f:N_1\cong N_2$ such that $f\restriction M_1:M_1\cong M_2$. 
\item Let $\lambda\geq\ls$. $I_2(\lambda,{\bf K})\defeq\{(M,N)/_{\cong}:M\leq N\text{ in }K_\lambda\}$.
\end{enumerate}
\end{definition}
\begin{example}\mylabel{i2ex}{Example \thetheorem}
Depending on $\lambda$ and ${\bf K}$, $I(\lambda,{\bf K})$ and $I_2(\lambda,{\bf K})$ may not be the same:
\begin{enumerate}
\item If ${\bf K}$ is the class of the $L_{\omega_1,\omega}$ theory $\forall x\ \bigvee_{i<\omega}(x=c_i)$ where $c_i$ are constants, then $I(\al,{\bf K})=I_2(\al,{\bf K})=1$.
\item If ${\bf K}$ is the class of the first-order theory of pure equality, then $I(\lambda,{\bf K})=1$ but $I_2(\lambda,{\bf K})=\lambda$ for any infinite $\lambda$.
\item If ${\bf K}$ is the class of the first-order theory of dense linear orders without endpoints, then $I(\aleph_0,{\bf K})=1$ but $I_2(\aleph_0,{\bf K})=2^{\al}$. The latter is because one can choose $M=\mathbb{Q}$ while $N_r=\{a+br:a,b\in\mathbb{Q}\}$ where $r$ is an irrational number.
\end{enumerate}
\end{example}
\begin{proposition}\mylabel{quesprop}{Proposition \thetheorem}
$I(\lambda,{\bf K})\leq I_2(\lambda,{\bf K})\leq 2^{\lambda}$.
\end{proposition}
\begin{proof}
For any $M\leq N_1$ and $M\leq N_2$ in $K$, if $N_1\not\cong N_2$, then $(M,N_1)\not\cong (M,N_2)$ by definition, hence the first inequality. Using the fact that $I(\lambda,{\bf K})\leq2^{\lambda}$, we can bound $I_2(\lambda,{\bf K})\leq \lambda^\lambda\cdot I(\lambda,{\bf K})\leq 2^\lambda$. 
\end{proof}
\begin{ques}
Assuming stability or categoricity, is it possible to obtain a better bound than \ref{quesprop}?
\end{ques}

Until the end of this section, we write $\lambda\defeq\ls$. 

\begin{obs}
List the representatives of $\{M/_{\cong}:M\in K_\lambda\}$ by $\langle M_i:i<I(\lambda,{\bf K})\rangle$ and those of $\{(M,N)/_{\cong}:M\leq N\text{ in }K_\lambda\}$ by $\langle (M_j,N_j):j<I_2(\lambda,{\bf K})\rangle$. For $i<I(\lambda,{\bf K})$, let $\phi_i(x)$ be an $L_{\lambda^+,\lambda^+}$ formula that encodes the isomorphism type of $M_i$ with a fixed enumeration of the universe $|M_i|=\{ m_k^i:k<\lambda\}$. For variables $x=\langle x_k:k<\lambda\rangle$,
\begin{align*}
\phi_i(x)\defeq\bigwedge\{&\theta(x_{\alpha_0},\dots,x_{\alpha_{s-1}}):M_i\vDash \theta[m_{\alpha_0}^i,\dots,m_{\alpha_{s-1}}^i],s<\omega, \alpha_0,\dots,\alpha_{s-1}<\lambda,\\&\theta\text{ is an atomic $L$-formula or its negation with $s$ free variables}\}
\end{align*}
Namely for any $L$-structure $N$ and any $a\in|N|$ of length $\lambda$, if $N\vDash\phi_i[a]$ then $a\cong M_i$ (with the fixed enumeration). 

Similarly, for $j<I_2(\lambda,{\bf K})$, let $\psi_j(x,y)$ be an $L_{\lambda^+,\lambda^+}$ formula that encodes the isomorphism type of $(M_j,N_j)$ with fixed enumerations, where $|M_j|=\langle m_k^j:k<\lambda\rangle$, $|N_j|=\{ n_k^j:k<\lambda\}$. For variables $x=\langle x_k:k<\lambda\rangle$ and $y=\langle y_k:k<\lambda\rangle$,
\begin{align*}
\psi_j(x;y)\defeq\bigwedge\{&\theta(x_{\alpha_0},\dots,x_{\alpha_{s-1}};y_{\beta_0},\dots,y_{\beta_{t-1}}):N_j\vDash \theta[m_{\alpha_0}^j,\dots,m_{\alpha_{s-1}}^j;n_{\beta_0}^j,\dots,n_{\beta_{t-1}}^j],\\& s<\omega,\ t<\omega,\ \alpha_0,\dots,\alpha_{s-1},\beta_0,\dots,\beta_{t-1}<\lambda,\\&\theta\text{ is an atomic $L$-formula or its negation with $s+t$ free variables}\}
\end{align*}
Namely for any $L$-structure $N$ and any $a,b\in|N|$ both of length $\lambda$, if $N\vDash\psi_j[a,b]$ then $a\cong M_j$, $b\cong N_j$ (with the fixed enumerations) and $\ran(a)\subseteq\ran(b)$.

It is also possible to encode the re-enumerations of the isomorphism types in $\phi_i$ and $\psi_j$, but we will do that directly in the sentence $\sigma_{\bf K}$ in \ref{charmod} and $\sigma_\leq$ in \ref{charsub}(1), so as to be more consistent with the format of \ref{presvar}.
\end{obs}
\begin{definition}
Let $\alpha,\beta<\lambda^+$, $a=\langle a_i:i<\alpha\rangle $ and $b=\langle b_i:i<\beta\rangle$. $a\subseteq b$ stands for $\ran(a)\subseteq\ran(b)$, which can be expressed by the $L_{\lambda^+,\lambda^+}$ formula $$\bigwedge_{i<\alpha}\bigvee_{j<\beta}a_i=b_j$$
$a\approx b$ stands for $\ran(a)=\ran(b)$, which can be expressed by the $L_{\lambda^+,\lambda^+}$ formula
$$\Big(\bigwedge_{i<\alpha}\bigvee_{j<\beta}a_i=b_j\Big)\wedge\Big(\bigwedge_{j<\beta}\bigvee_{i<\alpha}b_j=a_i\Big)$$
\end{definition}

\begin{mainthm}\mylabel{charmod}{Main Theorem \thetheorem}
$K$ is axiomatizable by an $L_{(\lambda+I_2(\lambda,{\bf K}))^+,\lambda^+}(\omega\cdot\omega)$ sentence $\sigma_{\bf K}$. In other words, for any $L$-structure $M$, $M\in K$ iff $M\vDash \sigma_{\bf K}$.
\end{mainthm}
\begin{proof}
The following variables $(x_\alpha,y_\alpha)_{\alpha<\omega\cdot\omega}$ are all of length $\lambda$.
\begin{align*}
\sigma_{\bf K}\defeq(\forall x_\alpha\exists y_\alpha)_{\alpha<\omega\cdot\omega}&\bigwedge_{\alpha<\omega\cdot\omega}\Big[(x_\alpha\subseteq y_\alpha)\wedge\exists z_\alpha\Big((y_\alpha\approx z_\alpha)\wedge\bigvee_{i<I(\lambda,{\bf K})}\phi_i(z_\alpha)\Big)\wedge\\
&\bigwedge_{\beta<\alpha<\omega\cdot\omega}\exists u_{\beta,\alpha}\exists v_{\beta,\alpha}\Big((y_\beta\approx u_{\beta,\alpha})\wedge(y_\alpha\approx v_{\beta,\alpha})\wedge\bigvee_{j<I_2(\lambda,{\bf K})}\psi_j(u_{\beta,\alpha},v_{\beta,\alpha})\Big)\Big]
\end{align*}
In words, $\sigma_{\bf K}$ stipulates the iterated use of L\"{o}wenheim-Skolem and coherence axioms $(\omega\cdot\omega)$ many times. 

Suppose $M\in K$, we show that Player II can win the associated game in $\sigma_{\bf K}$. In the $\alpha$-th round, Player I provides some $x_\alpha$ of length $\lambda$. By L\"{o}wenheim-Skolem axiom, pick any $y_\alpha\leq M$ of size $\lambda$ such that $\ran(x_\alpha)\cup\bigcup_{\beta<\alpha}\ran(y_\beta)\subseteq\ran(y_\alpha)$. By inductive hypothesis, for $\beta<\alpha$, we have $y_\beta\leq M$. By coherence axiom, $y_\beta\leq y_\alpha$ as desired.

Suppose $M\vDash\sigma_{\bf K}$. By \ref{dirfact}(1), it suffices to build a directed system $\langle M_a\in K_\lambda:a\in I\rangle$ of union $M$. We choose $I$ to be the set of finite tuples $a$ in $M$. Let $a,b$ be finite tuples in $M$, we order $a\leq_Ib$ iff $\ran(a)\subseteq\ran(b)$. We will inductively build all $M_a$ in $\omega$ stages. At stage $n$ we handle finite tuples of length $n+1$.

\begin{itemize}
\item Stage 0: apply $\sigma_{\bf K}$ to each singleton $s$ in $M$ and substitute $x_0=s$. We obtain $y_0=M_s$ which is a ${\bf K}$-structure. Only the $0$-th round of the game is used for each singleton.
\item Inductive hypothesis: for some $n<\omega$, $M_a$ has been constructed for each $l(a)\leq n+1$ with the following requirements:
\begin{enumerate}
\item For some $k<\omega$ and some singleton $s$, $M_a$ is the union $\bigcup_{\alpha<1+\omega\cdot k}y_\alpha$ where $y_\alpha$ comes from the game of $s$ (the ``$1+$'' is to handle the $k=0$ case). 
\item Let $a,b$ both of length $\leq n+1$. If $\ran(b)\subseteq\ran(a)$, then $M_b\leq M_a$.
\end{enumerate} 
\end{itemize}
Before we move onto the inductive step, we show that given $M_a$ and $M_b$ constructed in previous stages, we can find $M^*$ such that $M^*\geq M_a$ and $M^*\geq M_b$. By inductive hypothesis, there is a singleton $s$ in $M$, $m_s<\omega$ such that $M_a$ is the union $\bigcup_{\alpha<1+\omega\cdot m_s}y_\alpha$ from the game of $s$. Similarly, we can find some singleton $t$ in $M$ and $m_t<\omega$ such that $M_b$ is the union $\bigcup_{\alpha<1+\omega\cdot m_t}y_\alpha$ from the game of $t$. Using $\omega$ more rounds in the games of $s$ and of $t$, we recursively build $\langle N_k:k<\omega\rangle$ $\subseteq$-increasing such that 
\begin{enumerate}
\item $\langle N_{2l}:l<\omega\rangle$ and $\langle N_{2l+1}:l<\omega\rangle$ are both $\lk$-increasing. 
\item $N_{-1}\defeq M_a$ and $N_0\defeq M_b$. 
\item If $k=2l+1$, then use the $(1+\omega\cdot m_s+l)$-th round in the game of $s$ to obtain $y_{1+\omega\cdot m_s+l}=N_k$ from $x_{1+\omega\cdot m_s+l}=N_{k-1}$. Notice that if $m_s>0$, $N_{-1}\leq N_1$ by the second chain axiom.
\item If $k=2l+2$, then use the $(1+\omega\cdot m_t+l)$-th round in the game of $t$ to obtain $y_{1+\omega\cdot m_t+l}=N_k$ from $x_{1+\omega\cdot m_t+l}=N_{k-1}$. Notice that if $m_t>0$, $N_{0}\leq N_2$ by the second chain axiom.\\
\end{enumerate}
\adjustbox{scale=0.85,center}
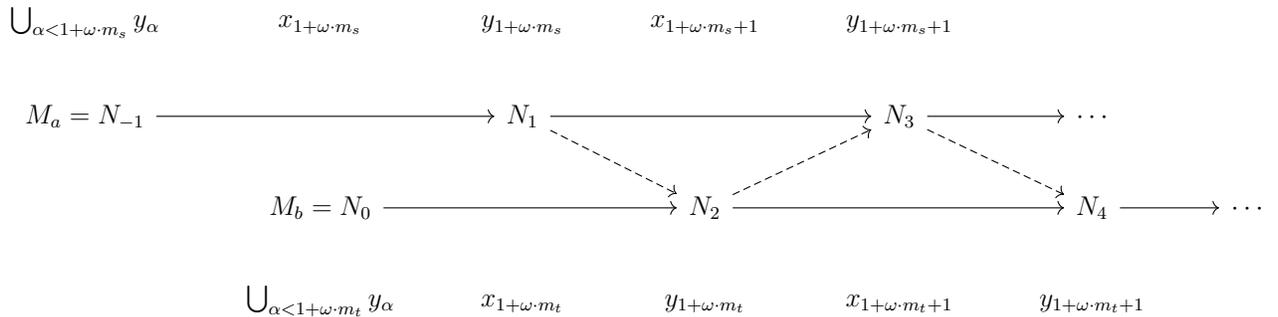
{
\begin{tikzcd}
\bigcup_{\alpha<1+\omega\cdot m_s}y_\alpha & x_{1+\omega\cdot m_s}                      & y_{1+\omega\cdot m_s}             & x_{1+\omega\cdot m_s+1}           & y_{1+\omega\cdot m_s+1}          &                         &        \\
M_a=N_{-1} \arrow[rr]                      &                                            & N_1 \arrow[rr] \arrow[rd, dashed] &                                   & N_3 \arrow[r] \arrow[rd, dashed] & \cdots                  &        \\
                                           & M_b=N_0 \arrow[rr]                         &                                   & N_2 \arrow[rr] \arrow[ru, dashed] &                                  & N_4 \arrow[r]           & \cdots \\
                                           & \bigcup_{\alpha<1+\omega\cdot m_t}y_\alpha & x_{1+\omega\cdot m_t}             & y_{1+\omega\cdot m_t}             & x_{1+\omega\cdot m_t+1}          & y_{1+\omega\cdot m_t+1} &       
\end{tikzcd}}
In the above diagram, a solid arrow stands for ${\bf K}$-substructure while a dashed arrow stands for $L$-substructure. The first row represents the game of $s$ while the last row represents the game of $t$. Each vertical column contains identical ${\bf K}$-structures.

Define $M^*\defeq\bigcup_{k<\omega}N_k$. By requirements (1), (2) and chain axioms, $M^*\geq M_a$ and $M^*\geq M_b$. Also, $s$ has used the first $\omega\cdot(m_s+1)$ rounds while $t$ has used the first $\omega\cdot(m_t+1)$ rounds. This finishes the construction of $M^*$.
\begin{itemize}
\item Stage $n+1$: Now for tuples $c$ of length $n+2$, we build $M_c$. Break down $c$ as the union of two tuples of length $\leq n+1$ (there might be more than one way), say $a$ and $b$. As above assume $M_a$ is generated by some singleton $s$ and $M_b$ by some singleton $t$. Then we can find $M^*$ with $M^*\geq M_a$,\ $M^*\geq M_b$ and $M^*$ is the union of bounded many $y_\alpha$'s of the game of some singleton $s$. We cannot immediately define $M_c\defeq M^*$ because $M^*$ depends on the choice of the decomposition $a,b$. Since there are finitely many possible decompositions of a finite tuple $c$, we can continue the game of $s$ and extend $M^*$ to $M_c$ which includes all $M^*$ from other decompositions of $c$ ($M_c$ might not be unique but it is a ${\bf K}$-superstructure to all those $M_a$ with $\ran(a)\subseteq\ran(c)$; $M_c$ is also generated by other games of singletons but we just need one representative $s$ for $M_c$).
\end{itemize}
After the construction is completed, $\langle M_a\in K_\lambda:a\in I\rangle$ is directed by our inductive step. Their union is $M$ because for each element $u$ in $M$, $u\in M_{\{u\}}$.
\end{proof}

\begin{remark}
Our theorem generalizes \cite[Theorems 2.9, 3.7]{kue2} which use the $\omega$-game quantification. Modulo the $(\omega\cdot\omega)$-game quantification, our result also generalizes \cite[Theorems 5.3, 7.4]{kue2} which assumes a monster model and categoricity in a higher cardinal.
\end{remark}

We will encode the ${\bf K}$-substructure relation in \ref{charsub}. Before that we investigate possible improvements of \ref{charmod}. The following questions were suggested by Grossberg:
\begin{ques}\mylabel{axiques}{Question \thetheorem}
Is it possible to axiomatize an AEC ${\bf K}$ in $L_{\lambda^+,\lambda^+}$ instead of $L_{(\lambda+I_2(\lambda,{\bf K}))^+,\lambda^+}$ (with or without game quantification), assuming
\begin{enumerate}
\item stability?
\item categoricity in $\lambda$ and $\lambda^+$? 
\end{enumerate}
\end{ques}
For (2), Grossberg also suggested that \cite{sm} should allow improvements of \ref{charmod}. Indeed it is possible when ${\bf K}$ is a universal class. A partial converse can be found in \cite[Theorem 3.5]{sm}.
\begin{proposition}\mylabel{charuni}{Proposition \thetheorem}
If ${\bf K}$ is a universal class, then it is axiomatizable by an $L_{(\lambda+I(\lambda,{\bf K}))^+,\lambda^+}$ sentence $\sigma_{\bf K}$. In particular if ${\bf K}$ is categorical in $\lambda$, then it is axiomatizable by an $L_{\lambda^+,\lambda^+}$ sentence.
\end{proposition}
\begin{proof}
Since models are ordered by $L$-substructures, we can avoid game quantification and replace the $\psi_j$'s by subset relations when defining $\sigma_{\bf K}$ in \ref{charmod}. Namely,
\begin{align*}
\sigma_{\bf K}\defeq\ &\forall a\ \exists m_a\Big(a\subseteq m_a\wedge\bigvee_{i<I(\lambda,{\bf K})}\phi_i(m_a)\Big)\wedge\\
&\forall m_b\,\forall m_c\Big[\Big(\bigvee_{i<I(\lambda,{\bf K})}\phi_i(m_b)\wedge\bigvee_{j<I(\lambda,{\bf K})}\phi_j(m_c)\Big)\\
&\rightarrow \exists m_d \Big(\bigvee_{k<I(\lambda,{\bf K})}\phi_k(m_d)\wedge (m_b\subseteq m_d)\wedge (m_c\subseteq m_d)\Big)\Big]
\end{align*}
\end{proof}

We also have some approximations to \ref{axiques}(2), using game quantification. In the following we abbreviate \emph{amalgamation property} as $AP$ and \emph{arbitrarily large models} as $AL$.
\begin{fact}\mylabel{axifact0}{Fact \thetheorem}
\begin{enumerate}
\item If ${\bf K}$ has $AL$, $\lambda$-$AP$, is categorical in $\lambda^+$, then it is stable in $\lambda$. Hence for any $M\in K_\lambda$, there is $N\in K_\lambda$ which is a $(\lambda,\omega)$-limit model over $M$.
\item Let ${\bf K}$ have $\lambda$-$AP$ and $M_1,M_2,M_3\in K_\lambda$. If $M_1\leq M_2$ and $M_3$ is a $(\lambda,\omega)$-limit model over $M_2$, then $M_3$ is also a $(\lambda,\omega)$-limit model over $M_1$.
\item Let $M,N,N'\in K_\lambda$. If $N,N'$ are both $(\lambda,\omega)$-limit models over $M$, then $N\cong_M N'$. 
\end{enumerate}
\end{fact}
\begin{theorem}\mylabel{charcat2}{Theorem \thetheorem}
If ${\bf K}$ has $AL$, $\lambda$-$AP$ and is cateogrical in $\lambda$ and $\lambda^+$, then it is axiomatizable by an $L_{\lambda^{+},\lambda^{+}}(\omega\cdot\omega)$ sentence $\sigma$. $AL$ can be replaced by stability in $\lambda$.
\end{theorem}
\begin{proof}
By \ref{axifact0}(1), fix $M,N\in K_\lambda$ such that $N$ is $(\lambda,\omega)$-limit over $M$. Let $\phi(x)$ code the isomorphism type of $M$ and $\psi(x,y)$ code the isomorphism type of $(M,N)$.
From the proof of \ref{charmod}, it suffices to replace $\bigvee_{i<I(\lambda,{\bf K})}\phi_i$ and $\bigvee_{j<I_2(\lambda,{\bf K})}\psi_j$ there by some $\lambda$-junctions. Since ${\bf K}$ is $\lambda$-categorical, we can replace the first disjunction by $\phi$. We also replace the second disjunction by a coherence formula involving $\psi$. Finally we add a disjunction to $\sigma_{\bf K}$ specifying models of size $\lambda$:
\begin{align*}
\sigma\defeq\ &\exists w\big(\phi(w)\wedge\forall x(x\subseteq w)\big)\vee\Big\{\\&
(\forall x_\alpha\exists y_\alpha)_{\alpha<\omega\cdot\omega}\,\bigwedge_{\alpha<\omega\cdot\omega}\Big[(x_\alpha\subseteq y_\alpha)\wedge\exists z_\alpha\Big((y_\alpha\approx z_\alpha)\wedge\phi(z_\alpha)\Big)\wedge\bigwedge_{\beta<\alpha<\omega\cdot\omega}\exists u_{\beta,\alpha}\exists v_{\beta,\alpha}\exists w\exists z\\
&\Big((u_{\beta,\alpha}\subseteq v_{\beta,\alpha})\wedge(y_\beta\approx u_{\beta,\alpha})\wedge(y_\alpha\approx v_{\beta,\alpha})\wedge(w\approx z)\wedge\psi(u_{\beta,\alpha},w)\wedge\psi(v_{\beta,\alpha},z))\Big)\Big]\Big\}
\end{align*}

Suppose $M^*$ is an $L$-structure and $M^*\vDash\sigma$. By coherence, the last line of $\sigma$ implies $y_\beta\leq y_\alpha$. Then either $M^*\in K_\lambda$ or it can build a directed system $\langle M_\alpha\in K_\lambda:\alpha\in I\rangle$ of union $M^*$ as in \ref{charmod}. By \ref{dirfact}(1) $M^*\in K$. 

Suppose $M^*\in K_\lambda$, then it satisfies the first disjunct of $\sigma$. Otherwise $M^*\in K_{>\lambda}$. We need to verify that for any $M_1,M_2\in K_\lambda$, if $M_1\leq M_2\leq M^*$ then there is $M_3\leq M^*$ such that $(M_1,M_3)\cong(M_2,M_3)\cong (M,N)$. By $AL$ and categoricity in $\lambda^+$, $M^*$ is $\lambda^+$-saturated, so we can build $M_3\in K_\lambda$ which is $(\lambda,\omega)$-limit over $M_2$. 
\begin{center}
\begin{tikzcd}
                                                                     M_3 \arrow[rr, "\cong"]             &                                                                                           & N' \arrow[r, "\cong"] & N                                               \\
                                                                     M_2 \arrow[u, "{(\lambda,\omega)}"] &                                                                                           &                       &                                                 \\
                                     & M_1 \arrow[lu] \arrow[luu, "{(\lambda,\omega)}"', dotted, bend right] \arrow[rr, "\cong"] &                       & M \arrow[uu, "{(\lambda,\omega)}"'] \arrow[luu]
\end{tikzcd}
\end{center} 
By \ref{axifact0}(2), $M_3$ is $(\lambda,\omega)$-limit to both $M_1$. Since ${\bf K}$ is categorical in $\lambda$, $M_1\cong M$ and we can extend this isomorphism to $M_3\cong N'$ for some $N'\geq M$. Then $N'$ is a $(\lambda,\omega)$-limit over $M$. By \ref{axifact0}(3), $N'\cong_MN$ so $(M_1,M_3)\cong (M,N)$. Similarly $(M_2,M_3)\cong (M,N)$. Therefore, $M^*\vDash \sigma$ as desired.
\end{proof}
Using a well-known result of Shelah, we can replace the assumption by $AP$ by a set-theoretic one.
\begin{fact}\mylabel{axifact}{Fact \thetheorem}\cite[Theorem 3.5]{sh88}
Assume $2^\lambda<2^{\lambda^+}$. If $I(\lambda,{\bf K})=1$ and $1\leq I(\lambda^+,{\bf K})<2^{\lambda^+}$, then ${\bf K}$ has $\lambda$-$AP$. 
\end{fact}
\begin{corollary}\mylabel{charcat2cor}{Corollary \thetheorem}
Assume $2^\lambda<2^{\lambda^+}$. If ${\bf K}$ has $AL$ and is cateogrical in $\lambda$ and $\lambda^+$, then it is axiomatizable by an $L_{\lambda^{+},\lambda^{+}}(\omega\cdot\omega)$ sentence $\sigma'$. $AL$ can be replaced by stability in $\lambda$.
\end{corollary}
\begin{proof}
Combine \ref{axifact} and \ref{charcat2}.
\end{proof}
In other words, under WGCH \ref{axiques} has a positive answer when we assume both (1) and (2) of the hypotheses there and use game quantification.
\begin{obs}
\begin{enumerate}
\item As in \ref{charuni}, we can replace categoricity in $\lambda$ by $I(\lambda,{\bf K})\leq\lambda$. We keep the original format of the theorem statements to better answer \ref{axiques}. 
\mylabel{charobs}{Observation \thetheorem}
\item Let $\kappa\geq\lambda^+$. In \ref{charcat2}, if we replace categoricity in $\lambda^+$ by $\kappa$, and further assume $(<\kappa)$-$AP$, then models in $K_\kappa$ are saturated \cite[Corollary 4.11(3)]{s17}. The same argument allows us to axiomatize $K_{\geq\kappa}$ by an $L_{\lambda^+,\lambda^+}(\omega\cdot\omega)$ sentence. If $\kappa$ is regular, then we can replace $AL$ by stability in $[\lambda,\kappa)$.
\item John Baldwin pointed out that \cite{sh394} can reduce the successive categoricity assumption to a single categoricity. Indeed, assuming $AP$ and categoricity in a successor above $H_2$ (the second Hanf number), we have categoricity in $H_2$ and $AL$. Thus we can axiomatize $K_{\geq H_2}$ by a sentence in $L_{H_2^+,H_2^+}(\omega\cdot\omega)$.
\item Marcos Mazari-Armida observed that \cite{sh576} provides another variation on \ref{charcat2cor}: if we assume categoricity in $\lambda$, $\lambda^+$, $\lambda^{++}$ as well as $2^\lambda<2^{\lambda^+}<2^{\lambda^{++}}$, then we have stability in $\lambda$ and $\lambda$-$AP$. Hence we can axiomatize $K$ by a sentence in $L_{\lambda^+,\lambda^+}(\omega\cdot\omega)$.
\end{enumerate}
\end{obs}

We now encode the ${\bf K}$-substructure relation. First we handle the case when the smaller model has size $\lambda$. In \cite[Section 6]{snote}, the language is expanded by adding a new predicate for the substructure relation. In \cite[Theorem 3.2.3]{bbhanf}, $I_2(\lambda,{\bf K})$ many new predicates are added. Here we explicitly define the predicates in $L_{(\lambda+I_2(\lambda,{\bf K}))^+,\lambda^+}(\omega\cdot\omega)$ without expanding the language. 

\begin{proposition}\mylabel{charprop}{Proposition \thetheorem}
There is an $L_{(\lambda+I_2(\lambda,{\bf K}))^+,\lambda^+}(\omega\cdot\omega)$ formula $\sigma_\leq(x)$ that encodes the ${\bf K}$-substructure relation: for any $M\in K$, $a\subseteq|M|$ of size $\lambda$, $M\vDash\sigma_\leq [a]$ iff $a\leq M$ (the enumeration of $a$ does not matter).
\end{proposition}
\begin{proof}
Our definition $\sigma_\leq(x)$ will be similar to that of $\sigma_{\bf K}$:
\begin{align*}
\sigma_{\leq}(x)\defeq(\forall x_\alpha\exists y_\alpha)&_{\alpha<\omega\cdot\omega}\ \Big(\exists u\exists v(x\approx u)\wedge(y_0\approx v)\wedge\bigvee_{j<I_2(\lambda,{\bf K})}\psi_j(u,v)\Big)\wedge\\&\bigwedge_{\alpha<\omega\cdot\omega}\Big[(x_\alpha\subseteq y_\alpha)\wedge\exists z_\alpha\Big((y_\alpha\approx z_\alpha)\wedge\bigvee_{i<I(\lambda,{\bf K})}\phi_i(z_\alpha)\Big)\wedge\\
&\;\bigwedge_{\beta<\alpha<\omega\cdot\omega}\exists u_{\beta,\alpha}\exists v_{\beta,\alpha}\Big((y_\beta\approx u_{\beta,\alpha})\wedge(y_\alpha\approx v_{\beta,\alpha})\wedge\bigvee_{j<I_2(\lambda,{\bf K})}\psi_j(u_{\beta,\alpha},v_{\beta,\alpha})\Big)\Big]
\end{align*}

The enumeration of $a$ does not matter by our definition of $\approx$. If $a\in K_\lambda$ and $a\leq M$, then $M\vDash \sigma_\leq[a]$ by L\"{o}wenheim-Skolem and coherence axioms. In particular, given $x_0$ in $\sigma_\leq[a]$, we choose $y_0\leq M$ that contains both $x_0$ and $a$. Then coherence guarantees that $a\leq y_0$. Conversely suppose $a\subseteq|M|$ of size $\lambda$ and $M\vDash\sigma_\leq [a]$. As in \ref{charmod}, we can build a directed system $\langle M_\alpha\in K_\lambda:\alpha\in I\rangle$ of union $M$ with the additional requirement that for any $\alpha\in I$, $M_\alpha\geq a$. By \ref{dirfact}(1)(2), $M\in K$ and $M_\alpha\leq M$ for all $\alpha\in I$. By transitivity of $\leq\ $, $a\leq M$ as desired. 
\end{proof}
\begin{remark}
Using the terminology in \cite[Theorem 2.1]{sv}, we only used the first two levels of the canonical tree because a $K$-substructure relation only concerns two levels. The price to pay is game quantification.
\end{remark}
Since our infinitary language only allows $\lambda$ many free variables, it cannot directly encode substructures of size greater than $\lambda$. We propose two solutions: the first solution is a substructure relation whose underlying language is the singleton $\{\sigma_\leq\}$. The second solution involves relativizing $\sigma_ \leq$ to a new predicate. 

\begin{proposition}\mylabel{charsub}{Proposition \thetheorem}
Let $M,N\in K$ and $\sigma_\leq$ be defined as in \ref{charprop}.
\begin{enumerate}
\item  $M\leq N$ iff $M\subseteq_{\{\sigma_\leq\}}N$ (if $a\subseteq |M|$ is of size $\lambda$, then $M\vDash \sigma_\leq[a] $ iff $N\vDash\sigma_\leq[a]$).
\item Let $R$ be a new predicate where $N^R=|M|$ closed under permutations.  $M\leq N$ iff $(N,R)\vDash\forall b\ \big(\sigma_\leq^R(b)\rightarrow\sigma_\leq(b)\big)$ where $\sigma_\leq^R$ is the relativized version of $\sigma_\leq$ inside $R$. Namely, 
\begin{align*}
\sigma^R_{\leq}(x)\defeq\ &(\forall x_\alpha\exists y_\alpha)_{\alpha<\omega\cdot\omega}\ \Big(\bigwedge_{\alpha<\omega\cdot\omega}R(x_\alpha)\Big)\rightarrow\Big\{\Big(\bigwedge_{\alpha<\omega\cdot\omega}R(y_\alpha)\Big)\wedge\\&\Big(\exists u\exists v(x\approx u)\wedge(y_0\approx v)\wedge\bigvee_{j<I_2(\lambda,{\bf K})}\psi_j(u,v)\Big)\wedge\\&\bigwedge_{\alpha<\omega\cdot\omega}\Big[(x_\alpha\subseteq y_\alpha)\wedge\exists z_\alpha\Big((y_\alpha\approx z_\alpha)\wedge\bigvee_{i<I(\lambda,{\bf K})}\phi_i(z_\alpha)\Big)\wedge\\
&\;\bigwedge_{\beta<\alpha<\omega\cdot\omega}\exists u_{\beta,\alpha}\exists v_{\beta,\alpha}\Big((y_\beta\approx u_{\beta,\alpha})\wedge(y_\alpha\approx v_{\beta,\alpha})\wedge\bigvee_{j<I_2(\lambda,{\bf K})}\psi_j(u_{\beta,\alpha},v_{\beta,\alpha})\Big)\Big]\Big\}
\end{align*}
\end{enumerate}
\end{proposition}
\begin{proof}
\begin{enumerate}
\item If $M\leq N$ and let $a\subseteq M$. Using \ref{charprop}, if $M\vDash \sigma_\leq[a]$, then $a\leq M\leq N$ showing $N\vDash\sigma_\leq[a]$. If $N\vDash\sigma_\leq[a]$, then $a\leq N$. By coherence, $a\leq M$ and so $M\vDash\sigma_\leq[a]$. Conversely, build a directed system $\langle M_\alpha\in K_\lambda:\alpha\in I\rangle$ inside $M$ such that for all $\alpha\in I$, $M_\alpha\leq M$. Then $M\vDash\sigma_\leq[M_\alpha]$. Since $M\subseteq_{\{\sigma_\leq\}}N$, we have $N\vDash\sigma_\leq[M_\alpha]$ and $M_\alpha\leq N$. The result follows from \ref{dirfact}(3).
\item If $M\leq N$ and $N\vDash \sigma_\leq^R[b]$ for some $b\subseteq|N|$, we need to show that $N\vDash \sigma_\leq[b]$. By assumption we can build a directed system of union $M$ and have $b\leq M$. By transitivity of $\leq\ $, $b\leq N$ and the conclusion follows. Conversely, by \ref{dirfact}(3), it suffices to build a directed system $\langle M_\alpha\in K_\lambda:\alpha\in I \rangle$ of union $M$ such that for all $\alpha\in I$, $M_\alpha\leq N$. Since $(N,R)\vDash\forall b\ \big(\sigma_\leq^R(b)\rightarrow\sigma_\leq(b)\big)$, we can require $M_\alpha\leq M$ instead of $M_\alpha\leq N$. Such construction is possible by L\"{o}wenheim-Skolem and coherence axioms.  
\end{enumerate}
\end{proof}
Using game quantification, we derive a simple proof to \cite[Theorem 7.2]{kue2}, which uses back-and-forth arguments. \cite[Theorem 6.21]{snote} proved similarly by transfering the AEC to its substructure expansion and translating results between $\equiv_{\infty,\lambda^+}$ and $\lambda^+$ back-and-forth systems.
\begin{corollary}\mylabel{vasapp}{Corollary \thetheorem}
Let $M,N$ be $L$-structures. If either $M$ or $N$ is in $K$ and $M\preceq_{L_{\infty,\lambda^+}(\omega\cdot\omega)}N$, then $M\leq N$ (and both are in $K$). 
\end{corollary}
\begin{proof}
Since $M\preceq_{L_{\infty,\lambda^+}(\omega\cdot\omega)}N$, $M\subseteq_{L_{(\lambda+I_2(\lambda,{\bf K}))^+,\lambda^+}(\omega\cdot\omega)}N$. In particular $M\vDash \sigma_{\bf K}$ iff $N\vDash\sigma_{\bf K}$. By \ref{charmod}, either $M,N$ is in $K$ implies both are in $K$. On the other hand, the assumption implies $M\subseteq_{\{\sigma_\leq\}}N$. By \ref{charsub}(1), $M\leq N$. 
\end{proof}

One can ask the same \ref{axiques} for ${\bf K}$-substructure relation instead of models of $K$. The following are variations on \ref{charprop}: 
\begin{corollary}\mylabel{charpropcor}{Corollary \thetheorem}
There is a formula $\sigma_\leq(x)$ in $L_{\lambda^{+},\lambda^{+}}(\omega\cdot\omega)$ that encodes the ${\bf K}$-substructure relation (for any $M\in K$, $a\subseteq|M|$ of size $\lambda$, $M\vDash\sigma_\leq [a]$ iff $a\leq M$), assuming one of the following:
\begin{enumerate}
\item ${\bf K}$ is a universal class, in which case $\sigma_\leq(x)$ is in $L_{\lambda^+,\lambda^+}$ (categoricity is not needed).
\item ${\bf K}$ has $AL$, $\lambda$-$AP$ and is cateogrical in $\lambda,\lambda^+$ and we restrict the use of $\sigma_\leq$ to models of size $\geq\lambda^+$. $AL$ can be replaced by stability in $\lambda$.
\item $2^\lambda<2^{\lambda^+}$, ${\bf K}$ has $AL$ and is cateogrical in $\lambda,\lambda^+$ and we restrict the use of $\sigma_\leq$ to models of size $\geq\lambda^+$. $AL$ can be replaced by stability in $\lambda$.
\end{enumerate}
\end{corollary}
\begin{proof}
\begin{enumerate}
\item $\sigma_\leq(x)$ requires that $x$ is closed under functions.
\item Combine the proofs of \ref{charcat2} and \ref{charprop}: from \ref{charprop}, it suffices to encode $a\leq b$ where both have length $\lambda$ (then we can replace $\bigvee_{j<I_2(\lambda,{\bf K})}\psi_j(a,b)$). Fix $M,N\in K_\lambda$ such that $N$ is $(\lambda,\omega)$-limit over $M$. Let $\phi(x)$ code the isomorphism type of $M$ and $\psi(x,y)$ code the isomorphism type of $(M,N)$. From the proof of \ref{charcat2}, $\psi$ is the only isomorphism type of pairs inside a $\lambda^+$-saturated model. Thus $a\leq b$ can be encoded as
$$a\subseteq b\wedge\exists z_0\ \exists z_1\ \exists a'\ \exists b'\big((a\approx a')\wedge(b\approx b')\wedge(z_0\approx z_1)\wedge\psi(a',z_0)\wedge\psi(b',z_1)\big).$$
\item Combine (2) and \ref{axifact}.
\end{enumerate}
\end{proof}
With the exact same proof as in \ref{charsub}, we can show:
\begin{corollary}\mylabel{charsubcor}{Corollary \thetheorem}
Let $M,N\in K$ and $\sigma_\leq$ be defined as in \ref{charpropcor}.
\begin{enumerate}
\item  $M\leq N$ iff $M\subseteq_{\{\sigma_\leq\}}N$ (if $a\subseteq |M|$ is of size $\lambda$, then $M\vDash \sigma_\leq[a] $ iff $N\vDash\sigma_\leq[a]$).
\item Let $R$ be a new predicate where $N^R=|M|$ closed under permutations.  $M\leq N$ iff $(N,R)\vDash\forall b\ \big(\sigma_\leq^R(b)\rightarrow\sigma_\leq(b)\big)$ where $\sigma_\leq^R$ is the relativized version of $\sigma_\leq$ inside $R$.
\end{enumerate}
\end{corollary}
\section{A variation on Shelah's presentation theorem}

We will give a variation to Shelah's presentation theorem via our encoding of AECs. The presentation theorem statement is adapted from \cite[Theorem 4.15]{bal} (see also \cite[Theorem 3.4]{Gclass}).
\begin{theorem}\mylabel{presvar}{Theorem \thetheorem}
Let ${\bf K}$ be an AEC in $L$ and with L\"{o}wenheim-Skolem number $\ls$. Define $\chi\defeq\ls+I_2(\ls,{\bf K})$. There exists an expansion $L'\supseteq L$ of size $\chi$, an $L'$-theory $T$ and a set of $L'$-types $\Gamma$ of size $\chi$ such that 
\begin{enumerate}
\item $K=PC(T,\Gamma,L)$.
\item If $M',N'\in EC(T,\Gamma)$ and $M'\subseteq_{L'} N'$, then $M'\restriction L\lk N'\restriction L$.
\item If $M\lk N$, there are $L'$-expansions of $M,N$ to $M',N'$ such that $M'\subseteq_{L'}N'$. 
\end{enumerate}
\end{theorem}
\begin{remark}
When $I_2(\ls,{\bf K})<2^{\ls}$, our result is stronger than Shelah's presentation theorem as his encoding totally ignores the model theoretic complexity of ${\bf K}$ and is using $2^{\ls}$ many types in $\Gamma$.
\end{remark}
\begin{proof}
We will adapt $\sigma_{\bf K}$ in \ref{charmod} to the first-order language. For a tuple of variables $a$ and $n<\omega$, we write $a^n$ to emphasize $a$ has length $n$. For $k<n$, we write $a^n(k)$ the $k$-th coordinate of $a^n$.

As in the original presentation theorem, expand $L$ to $L'$ which includes $\{f_k^n:k<\ls,n<\omega\}$ where for each $n<\omega$, $k<\ls$, $f_k^n$ is an $n$-ary function. For $n<\omega$, we will require that $\{f_k^n:k<\ls\}$ maps an $n$-tuple to a $K$-structure of size $\ls$ containing that tuple. This will be achieved by 
$$\sigma^n\defeq\forall a^n\ \bigwedge_{l<n}\bigvee_{k<\ls}\big(f^n_k(a^n)=a^n(l)\big)\wedge\bigvee_{i<I(\lambda,{\bf K})}\phi_i\big(\{f^n_k(a^n):k<\ls\}\big)$$
In the above definition, although $\phi_i$'s have $\ls$ many free variables, it is just an $\ls$-conjunction of (negation of) atomic formulas with $n$ free variables (from $a^n$). So each $\phi_i$ is inside $L'_{{\ls^+},\omega}$. 

Also, we want to require that the $K$-structures generated are directed with respect to the tuple input. However, $\{f_k^n:k<\ls,n<\omega\}$ might not be compatible with the enumerations of \emph{pairs} of models, say $M_j\leq N_j$. Hence we expand $L'$ further to include $\{g_k^{m,l},h_k^{m,l}:k<\ls,m+l<\omega\}$ where for $m+l<\omega$, $k<\ls$, $g_k^{m,l}$ and $h_k^{m,l}$ are $(m+l)$-ary functions and correctly enumerate a pair of models. The following will take care of the re-enumerations of $\{f_k^n:k<\ls\}$ for each $n<\omega$.
\begin{align*}
\sigma^{m,n,l}\defeq\ &\forall b^m\ \forall c^n\ \forall d^l\Big[\ran(b^m)\cup\ran(c^n)\subseteq\ran(d^l)\rightarrow\\
&\Big(\{f^{m}_k(b^m):k<\ls\}\approx\{g^{m,l}_k(b^m;d^l):k<\ls\}\wedge\\
&\;\;\{f^{l}_k(d^l):k<\ls\}\approx\{h^{m,l}_k(b^m;d^l):k<\ls\}\wedge\\
&\;\;\{f^{n}_k(c^n):k<\ls\}\approx\{g^{n,l}_k(c^n;d^l):k<\ls\}\wedge\\
&\;\;\{f^{l}_k(d^l):k<\ls\}\approx\{h^{n,l}_k(c^n;d^l):k<\ls\}\wedge\\
&\;\;\bigvee_{i<I_2(\ls,{\bf K})}\psi_i\big(\{g^{m,l}_k(b^m;d^l):k<\ls\},\{h^{m,l}_k(b^m;d^l):k<\ls\}\big)\wedge\\ &\;\bigvee_{j<I_2(\ls,{\bf K})}\psi_j\big(\{g^{n,l}_k(c^n;d^l):k<\ls\},\{h^{n,l}_k(c^n;d^l):k<\ls\}\big)\Big)\Big]
\end{align*}
Similar to the case of $\sigma^n$, the formulas $\phi_i,\psi_j$ and the connective $\approx$ are simply $\ls$-junctions of (negation of) atomic formulas, which are inside $L'_{{\ls^+},\omega}$. 

To convert $\{\sigma^n:n<\omega\}\cup\{\sigma^{m,n,l}:m,n,l<\omega\}$ into first-order sentences, we use Chang's presentation theorem (see \cite[Chapter 1 Theorem 8.16]{Gbook}) which adds $\chi$-many new predicates to $L'$ to represent the $\chi$-conjunctions and disjunctions, and $\chi$-many $L'$-types to omit. This gives our final $T$, $\Gamma$ and $L'$. 

It remains to check the three items in the theorem statement. \begin{enumerate}
\item $K\subseteq PC(T,\Gamma,L)$ by L\"{o}wenheim-Skolem and coherence axioms. Let $M\in PC(T,\Gamma,L)$ and $M'\in EC(T,\Gamma)$ such that $M'\restriction L=M$. Then $M'\restriction L$ is the union of a directed system of $K$-structures of size $\ls$. By \ref{dirfact}(1) $M\in K$. Hence $PC(T,\Gamma,L)\subseteq K$.
\item By \ref{dirfact}(3). 
\item By \ref{dirfact2}.
\end{enumerate}
\end{proof}
Another question raised by Grossberg is the following:
\begin{ques}
Is it possible to lower the bound of $|\Gamma|$ below $\ls+I_2(\ls,{\bf K})$ in general? What if we also assume tameness or stability?
\end{ques}
\begin{remark}
\begin{enumerate}
\item In the above proof, we did not use $\sigma_{\leq}$ because $\{f_k^n:k<\ls, n<\omega\}$ already plays its role. We could have done the same in \ref{charsub} but the approach via $\sigma_{\leq}$ is cleaner and does not add new function symbols.
\item One might want to encode $\bigvee_{i<I_2(\lambda,{\bf K})}\psi_i$ etc by omitting types without raising $|T|$ above $\ls$. However, it amounts to list all pairs of isomorphism types that are not any of the $\psi_i$'s. This will raise $|\Gamma|$ to $2^{\ls}$ which is equivalent to the original presentation theorem.
\item In \cite[Theorem 3.2.3]{bbhanf}, new predicates are essentially added for our $\phi_i$ and $\psi_j$. Also, the requirement in \cite[Definition 3.2.1(4)]{bbhanf} encodes our $\{\sigma^{m,n,l}:m,n,l<\omega\}$ to build a directed system, but our version is purely syntactic and more direct. We used Chang's representation theorem to bring down the infinitary logic to first-order with omitting types. In \cite{bbhanf}, their theory $T^*$ is still in the infinitary logic and thus does not need to omit types. 
\end{enumerate}
\end{remark}

We derive the known results in the following corollary. Item (1) appeared for the first time in \cite{Gbook} while (2) and (3) were undoubtedly known to Chang \cite{chang}.

\begin{corollary}\mylabel{presvarcor}{Corollary \thetheorem}
Let ${\bf K}$ be an AEC and $\chi\defeq\ls+I_2(\ls,{\bf K})$.
\begin{enumerate}
\item $K$ is a $PC_{\chi}(=PC_{\chi,1})$ class. 
\item The Hanf number of ${\bf K}$ is bounded above by $\beth_{\delta(\chi,1)}$.
\item If $\chi=\al$ or $\chi$ is a strong limit with $\cf(\chi)=\al$, then the Hanf number of ${\bf K}$ is bounded above by $\beth_{\chi^+}$.
\end{enumerate}
\end{corollary}
\begin{proof}
\begin{enumerate}
\item Combine \ref{presvar} and Shelah's 1-type coding \cite[VII Lemma 5.1(4)]{sh90}.
\item Combine (1) and \cite[VII Theorem 5.3]{sh90}.
\item Combine (2) and the fact that $\delta(\chi,1)=\chi^+$ \cite[VII Theorem 5.5(5)]{sh90}.
\end{enumerate}
\end{proof}
We finish this section with one more application.

\begin{definition}
Let ${\bf K}$ be an AEC. ${K^<}\defeq\{\langle |M|,|N|\rangle:N<M\}$ is a class of structures whose language consists of a single unary predicate. 
\end{definition}

In 1994, motivated by \cite[Theorem 3.8]{sh88}, Grossberg suggested the following conjecture (see Problem (5) in \cite[Introduction]{sh576}):
\begin{conj}
Let ${\bf K}$ be an AEC, $\lambda\geq\ls$. If $I(\lambda,{\bf K})=I(\lambda^+,{\bf K})=1$, then $K_{\lambda^{++}}\neq\emptyset$.
\end{conj}
\cite[Theorem 3.8]{sh88} has two additional hypotheses:
\begin{enumerate}
\item Both $K$ and $K^<$ are $PC_\lambda$; and
\item $\delta(\lambda,1)=\lambda^+$.
\end{enumerate}
Much of \cite{sh576} is dedicated to the special cases of Grossberg's conjecture under various strong assumptions (including non-ZFC axioms).

Here we delete hypothesis (1) above and work in ZFC. In addition to hypothesis (2), we assume that $\lambda\geq\ls+I_2(\ls,{\bf K})$. 
\begin{theorem}\mylabel{presvarsuc}{Theorem \thetheorem}
Let ${\bf K}$ be an AEC, $\chi\defeq\ls+I_2(\ls,{\bf K})$, $\lambda\geq\chi$ with $\delta(\lambda,1)=\lambda^+$. If ${\bf K}$ is categorical in $\lambda$ and $\lambda^+$, then ${\bf K}$ has a model of cardinality $\lambda^{++}$.
\end{theorem}
\begin{remark}
Our theorem applies to the case $\ls=\aleph_1$, $\lambda=\beth_\omega$ while \cite[Theorem 3.8]{sh88} cannot handle uncountable $\llk$ or $\ls$.
\end{remark}
\begin{proof}
We check that hypothesis (1) above is satisfied. Since $\lambda\geq\chi$,  it suffices to show that both $K$ and $K^<$ are $PC_\chi$ classes. $K$ is a $PC_\chi$ class by \ref{presvarcor}. To show that $K^<$ is also a $PC_\chi$ class, we will use the proof of \ref{presvar}, add a new predicate $R$ in $L'$ and encode \ref{charsub}(2) by the new functions $\{f_k^n:k<\ls,n<\omega\}$ (to lighten the notation, we omit the encoding of re-enumerations, but it is the same strategy as in \ref{presvar}). At the end, we will only leave $R$ in the reduct of the language.

The details are as follows: we expand the language $L$ to $L'$ which includes a new predicate $R$ and the functions $\{f_k^n:k<\ls,n<\omega\}$ as in \ref{presvar}. For $n<\omega$, we abbreviate $\{f_k^n:k<\ls\}$ as $\bar{f}^n$ and require that it maps an $n$-tuple to a model of size $\ls$ containing the tuple. This can be achieved by
$$\sigma^n\defeq\forall a^n\ \bigwedge_{l<n}\bigvee_{k<\ls}\big(f^n_k(a^n)=a^n(l)\big)\wedge\bigvee_{i<I(\lambda,{\bf K})}\phi_i\big(\bar{f}^n(a^n)\big)$$
We also require that given an $n$-tuple inside $R$, the model generated is within $R$:
$$\sigma^n_R\defeq\forall a^n\subseteq R\big(\bar{f}^n(a^n)\subseteq R\big)$$
Next, we want to require that the models generated are directed with respect to the tuple input. For $m,n,l<\omega$, 
\begin{align*}
\sigma^{m,n,l}\defeq\ &\forall b^m\ \forall c^n\ \forall d^l\Big[\ran(b^m)\cup\ran(c^n)\subseteq\ran(d^l)\rightarrow\\
&\Big(\bigvee_{i<I_2(\lambda,{\bf K})}\psi_i\big(\bar{f}^m(b^m),\bar{f}^l(d^l)\big)\wedge\bigvee_{j<I_2(\lambda,{\bf K})}\psi_j\big(\bar{f}^n(c^n),\bar{f}^l(d^l)\Big)\Big]
\end{align*}
The final requirement is that $R$ is a proper subset of the model:
$$\sigma_p\defeq\exists x(\neg R(x))$$
Notice that for $m,n,l<\omega$, the sentences $\sigma^n,\sigma^n_R,\sigma^{m,n,l}$ are $\chi$-junctions of (negation of) atomic formulas, so we can use Chang's presentation theorem to convert them to first-order formulas, by adding $\chi$-many new predicates to $L'$ to represent the $\chi$-conjunctions and disjunctions, and $\chi$-many $L'$-types to omit. This gives our $T$, $\Gamma$ and $L'$. 

We check that $K^<=PC(T,\Gamma,\{R\})$. If $\langle|M|,|N|\rangle$ is in $K^<$, then $N<M$ are in $K$. Expand the language to $L'$ and define $R^M\defeq N$. Inside $N$, build a directed system of ${\bf K}$-substructures of size $\ls$, indexed by the finite tuples of $N$. This determines $\bar{f}^n\restriction R^n$ for $n<\omega$. Now inside $M$, we extend the directed system to be indexed by the finite tuples of $M$. This determines $\bar{f}^n$ completely for $n<\omega$. Also, $M$ satisfies $\sigma^n$, $\sigma^n_R$ and $\sigma^{m,n,l}$ for $m,n,l<\omega$. Hence $M$ under the expanded language is in $EC(T,\Gamma)$ and its reduct to $\{R\}$ is in $PC(T,\Gamma,\{R\})$.
Conversely, if $M\in PC(T,\Gamma,\{R\})$, expand $M$ to $M'$ such that $M'\in EC(T,\Gamma)$ and define $N'\defeq R^{M'}$. By $\{\sigma^n,\sigma^n_R,\sigma^{m,n,l}:m,n,l<\omega\}$, $N'\restriction L$ is the union of the directed system of ${\bf K}$-structures of size $\ls$. By \ref{dirfact}(1), $N'\restriction L\in K$.  By $\{\sigma^n,\sigma^{m,n,l}:m,n,l<\omega\}$, the directed system can be extended to union $M'\restriction L$. By \ref{dirfact}(1) again, $M'\restriction L\in K$. By \ref{dirfact}(2), each ${\bf K}$-structure of the directed system is a ${\bf K}$-substructure $M'\restriction L$. But then the models of the original system that generates $N'\restriction L$ are all ${\bf K}$-substructures of $M'\restriction L$. By \ref{dirfact}(3), $N'\restriction L\leq M'\restriction L$. By $\sigma_p$, $N'\restriction L<M'\restriction L$. In other words, $\langle |M'|,|N'|\rangle=\langle |M|,R^M\rangle=M\in K^<$.
\end{proof}
As in Section 3, we can add extra assumptions to improve our results:

\begin{corollary}
\begin{enumerate}
\item If ${\bf K}$ is a universal class, then $\chi\defeq\ls+I_2(\ls,{\bf K})$ in \ref{presvar}, \ref{presvarcor} and \ref{presvarsuc} can be replaced by $\chi\defeq{\ls}$. \label{presvarcot}
\item If ${\bf K}$ is cateogrical in ${\ls}$ and ${\ls}^+$, has $AL$ and either \begin{enumerate}\item ${\bf K}$ has $\ls$-$AP$; or \item $2^{\ls}<2^{{\ls}^+}$ \end{enumerate} then $K$ and $K^<$ are both $PC_{\ls}$ classes when restricted to models of size $\geq{\ls}^+$. 
In either case, $AL$ can be replaced by stability in $\ls$.
\item In (2), $K$ can be made to a $PC_{\ls}$ class.
\end{enumerate}
\end{corollary}
\begin{proof}[Proof sketch]
\begin{enumerate}
\item Combine the proof of \ref{presvar}, \ref{presvarcor} and \ref{presvarsuc} with \ref{charuni}. The point is that we do not need to encode ${\bf K}$-substructure relation.
\item In \ref{charcat2} and \ref{charcat2cor}, we used coherence to encode $a\leq b$ in the infinitary language: $a\subseteq b$ and there is $c$ which is $(\ls,\omega)$-limit over $a$ and $b$. Thus we can add two more sets of functions $\{d^{m,l}_k,e^{m,l}_k:k<\ls,m+l<\omega\}$ (which represent $c$) in addition to the original $\{g_k^{m,l},h_k^{m,l}:k<\ls,m+l<\omega\}$ (which represent $a$ and $b$) in \ref{presvar}.
\item It remains to handle the case when the models are of size $\ls$: add a disjunct to the theory $T$ in \ref{presvar} which stipulates that the model is generated by an element $a^1$:
$$\exists a^1\Big(\forall x^1\bigvee_{k<\ls}(f_k^1(a^1)=x^1)\wedge\phi\big(\{f_k^1(a^1):k<\ls\}\big)\Big).$$
\end{enumerate}
\end{proof}

\begin{remark}
\begin{enumerate}
\item \ref{charobs} also applies to the above corollary. If $\ls=\al$ in (3), then we obtain: if ${\bf K}$ is stable in $\al$, has $\al$-$AP$, $I(\al,{\bf K})\leq\al$ and $I(\aleph_1,{\bf K})=1$, then $K$ is $PC_{\al}$. This special case (with the extra assumption of categoricity in $\aleph_1$) provides an alternative proof to \cite[Theorem 4.2]{ss} which uses results from descriptive set theory.
\item We do not know if (3) also applies to $K^<$, for a similar reason in \ref{charpropcor}(2),(3).
\end{enumerate}
\end{remark}

\section{Generalization to $\mu$-AECs}

Our strategy of encoding AECs is also applicable to $\mu$-AECs. 
\begin{definition}\cite[Definitions 2.1,3.1]{muaec}\mylabel{chap2ques2}{Definition \thetheorem}
Let $L$ be a $(<\mu)$-ary language. A $\mu$-AEC ${\bf K}=\langle K,\lk\rangle$ in $L$ satisfies the axioms (1)(2)(3)(4) in \ref{aecdef} in addition to
\renewcommand{\labelenumi}{\alph{enumi}.}
\begin{enumerate}
\item Directed system axioms: if $\langle M_i\in K:i\in I\rangle$ is a $\mu$-directed system, then $M\defeq\bigcup_{i\in I}M_i\in K$ and for all $i\in I$, $M_i\lk M$. If in addition $N\in K$ with $M_i\lk N$ for all $i\in I$, then $M\lk N$.
\item L\"{o}wenheim-Skolem axiom: there exists a cardinal $\lambda=\lambda^{<\mu}\geq|\llk|+\mu$ such that for any $M\in K$, $A\subseteq|M|$, there is $N\lk M$ such that $A\subseteq|N|$ with $\nr{N}\leq|A|^{<\mu}+\lambda$. We call the minimum such cardinal the L\"{o}wenheim-Skolem number $\ls$. 
\end{enumerate}
\end{definition}

The analogs of \ref{charmod}, \ref{charprop} and \ref{charsub} hold in $\mu$-AECs.
\begin{proposition}\mylabel{charmod2}{Proposition \thetheorem}
Let ${\bf K}$ be a $\mu$-AEC in $L$ and $\lambda\defeq\ls$. $K$ is axiomatizable by an $L_{(\lambda+I_2(\lambda,{\bf K}))^+,\lambda^+}(\mu\cdot\mu)$ sentence $\sigma_{\bf K}$. In other words, for any $L$-structure $M$, $M\in K$ iff $M\vDash \sigma_{\bf K}$.
\end{proposition}
\begin{proof}
Similar to the proof in \ref{charmod}. The difference is that we allow the iterated use of L\"{o}wenheim-Skolem and coherence axioms $(\mu\cdot\mu)$-many times instead of $(\omega\cdot\omega)$-many.  We give the details below:

As usual, list the isomorphism types $\{M/_{\cong}:M\in K_\lambda\}$ by $\langle M_i:i<I(\lambda,{\bf K})\rangle$ and those of $\{(M,N)/_{\cong}:M\leq N\text{ in }K_\lambda\}$ by $\langle (M_j,N_j):j<I_2(\lambda,{\bf K})\rangle$. For $i<I(\lambda,{\bf K})$, let $\phi_i(x)$ be an $L_{\lambda^+,\lambda^+}$ formula that encodes the isomorphism type of $M_i$ with a fixed enumeration of the universe $|M|=\langle m_k^i:k<\lambda\rangle$. For variables $x=\langle x_k:k<\lambda\rangle$,
\begin{align*}
\phi_i(x)\defeq\bigwedge\{&\theta(x_{\alpha_0},\dots,x_{\alpha_{\xi}}):M_i\vDash \theta[m_{\alpha_0}^i,\dots,m_{\alpha_{\xi}}^i];\ \xi<\mu,\ \alpha_0,\dots,\alpha_{\xi}<\lambda,\\&\theta\text{ is an atomic $L$-formula or its negation with $s$ free variables}\}
\end{align*}
Notice that $\phi_i$ is a conjunction of $\lambda^{<\mu}=\lambda$ many formulas so it is inside $L_{{\lambda}^+,{\lambda}^+}$. Similarly for $j<I_2(\lambda,{\bf K})$, let $\psi_j(x,y)$ be an $L_{\lambda^+,\lambda^+}$ formula that encodes the isomorphism type of $(M_j,N_j)$ with fixed enumerations, where $|M_j|=\{ m_k^j:k<\lambda\}$, $|N_j|=\{ n_k^j:k<\lambda\}$. For variables $x=\langle x_k:k<\lambda\rangle$ and $y=\langle y_k:k<\lambda\rangle$,
\begin{align*}
\psi_j(x;y)\defeq\bigwedge\{&\theta(x_{\alpha_0},\dots,x_{\alpha_{\xi}};y_{\beta_0},\dots,y_{\beta_{\xi'}}):N_j\vDash \theta[m_{\alpha_0}^j,\dots,m_{\alpha_{\xi}}^j;n_{\beta_0}^j,\dots,n_{\beta_{\xi'}}^j],\\& \xi,\xi'<\mu;\ \alpha_0,\dots,\alpha_{\xi},\beta_0,\dots,\beta_{\xi'}<\lambda,\\&\theta\text{ is an atomic $L$-formula or its negation with $\xi+\xi'$ free variables}\}
\end{align*}
$\psi_j$ is also a conjunction of $\lambda^{<\mu}=\lambda$ many formulas so it is inside $L_{{\lambda}^+,{\lambda}^+}$. The axiomatization $\sigma_{\bf K}$ consists of two components (the variables all have length $\lambda$):
\begin{align*}
\sigma_{\bf K}\defeq(\forall x_\alpha\exists y_\alpha)_{\alpha<\mu\cdot\mu}\,&\bigwedge_{\alpha<\mu\cdot\mu}\Big[(x_\alpha\subseteq y_\alpha)\wedge\exists z_\alpha\Big((y_\alpha\approx z_\alpha)\wedge\bigvee_{i<I(\lambda,{\bf K})}\phi_i(z_\alpha)\Big)\wedge\\
&\bigwedge_{\beta<\alpha<\mu\cdot\mu}\exists u_{\beta,\alpha}\exists v_{\beta,\alpha}\Big((y_\beta\approx u_{\beta,\alpha})\wedge(y_\alpha\approx v_{\beta,\alpha})\wedge\bigvee_{j<I_2(\lambda,{\bf K})}\psi_j(u_{\beta,\alpha},v_{\beta,\alpha})\Big)\Big]
\end{align*}

Suppose $M\in K$, we show that Player II can win the associated game in $\sigma_{\bf K}$. In the $\alpha$-th round, Player I provides some $x_\alpha$ of length $\lambda$. By L\"{o}wenheim-Skolem axiom, pick any $y_\alpha\leq M$ of size $\lambda$ such that $\ran(x_\alpha)\cup\bigcup_{\beta<\alpha}\ran(y_\beta)\subseteq\ran(y_\alpha)$. By inductive hypothesis, for $\beta<\alpha$, we have $y_\beta\leq M$. By coherence axiom, $y_\beta\leq y_\alpha$ as desired.

Suppose $M\vDash\sigma_{\bf K}$. We will build a $\mu$-directed system $\langle M_a\in K_\lambda:a\in I\rangle$ of union $M$, with $I$ being the set of tuples of length $(<\mu)$ in $M$, ordered by inclusion. By directed system axioms, $M\in K$. In the proof of \ref{charmod}, we showed that given $M_a$ and $M_b$ generated by the games of the singletons $s$ and $t$, it is possible to find $M^*\geq M_a,M_b$ by extending those games by $\omega$-many rounds. In the $\mu$-AEC case, without the usual chain axioms we do not know if $M^*\in K$, so we extend those games by $\mu$-many rounds instead to obtain an increasing (but not necessarily continuous) chain $\langle N_k:k<\mu\rangle$ and define $M^*=\bigcup_{k<\mu}N_k\in K$. A similar argument shows: let $\delta<\mu$ and if for $\alpha<\delta$,  $M_\alpha$ is generated by the game of some tuple $a_\alpha$ of length $<\mu$, then we can extend the games by $\mu$-many rounds to obtain $M^*$ that extends all $M_\alpha$. This allows us to get past the limit stages which were absent in the original proof, and continue to build $M_a$ for $l(a)<\mu$. Given a tuple $c$ of length $<\mu$, there are less than $\mu$-many ways to decompose $c$ into a union of a singleton and a tuple of length $<\mu$. Thus we can still combine all copies of $M^*$ from the decompositions of $M_c$ as in the original proof.
\end{proof}
\begin{proposition}\mylabel{charprop2}{Proposition \thetheorem}
Let ${\bf K}$ be a $\mu$-AEC in $L$ and $\lambda\defeq\ls$. There is a formula in $L_{(\lambda+I_2(\lambda,{\bf K}))^+,\lambda^+}(\mu\cdot\mu)$ that encodes the ${\bf K}$-substructure relation: for any $M\in K$, $a\subseteq|M|$ of size $\lambda$, $M\vDash\sigma_\leq [a]$ iff $a\leq M$ (the enumeration of $a$ does not matter).
\end{proposition}
\begin{proof}
Define $\sigma_{\leq}(x)$ as in \ref{charprop} but replace $\omega\cdot\omega$ by $\mu\cdot\mu$. The enumeration of $a$ does not matter by our definition of $\approx$. If $a\in K_\lambda$ and $a\leq M$, then $M\vDash \sigma_\leq[a]$ by L\"{o}wenheim-Skolem and coherence axioms. Conversely suppose $a\subseteq|M|$ of size $\lambda$ and $M\vDash\sigma_\leq [a]$. As in \ref{charmod2}, we can build a $\mu$-directed system $\langle M_\alpha\in K_\lambda:\alpha\in I\rangle$ of union $M$ such that for any $\alpha\in I$, $M_\alpha\geq a$. By directed system axioms, $M\in K$ and $M_\alpha\leq M$ for all $\alpha\in I$. By transitivity of $\leq\ $, $a\leq M$ as desired. 
\end{proof}
\begin{proposition}\mylabel{charsub2}{Proposition \thetheorem}
Let ${\bf K}$ be a $\mu$-AEC in $L$ and $\lambda\defeq\ls$. Let $M,N\in K$.
\begin{enumerate}
\item  $M\leq N$ iff $M\subseteq_{\{\sigma_\leq\}}N$ (if $a\subseteq |M|$ is of size $\lambda$, then $M\vDash \sigma_\leq[a] $ iff $N\vDash\sigma_\leq[a]$).
\item Let $R$ be a new predicate where $N^R=|M|$ closed under permutations.  $M\leq N$ iff $(N,R)\vDash\forall b\ \big(\sigma_\leq^R(b)\rightarrow\sigma_\leq(b)\big)$ where $\sigma_\leq^R$ is the relativized version of $\sigma_\leq$ inside $R$ (replace $(\omega\cdot\omega)$ by $(\mu\cdot\mu)$ in the definition of $\sigma_\leq^R$ in \ref{charsub}).
\end{enumerate}
\end{proposition}
\begin{proof}
Similar to the proof in \ref{charsub}. The difference is that instead of building $\al$-directed systems, we build $\mu$-directed systems. We give details below:

\begin{enumerate}
\item If $M\leq N$ and let $a\subseteq M$. If $M\vDash \sigma_\leq[a]$, then $a\leq M\leq N$ showing $N\vDash\sigma_\leq[a]$. If $N\vDash\sigma_\leq[a]$, then $a\leq N$. By coherence, $a\leq M$ and so $M\vDash\sigma_\leq[a]$. Conversely, build a $\mu$-directed system $\langle M_\alpha\in K_\lambda:\alpha\in I\rangle$ inside $M$ such that for all $\alpha\in I$, $M_\alpha\leq M$. Then $M\vDash\sigma_\leq[M_\alpha]$. Since $M\subseteq_{\{\sigma_\leq\}}N$, we have $N\vDash\sigma_\leq[M_\alpha]$ and $M_\alpha\leq N$. The result follows from directed system axioms.
\item If $M\leq N$ and $N\vDash \sigma_\leq^R[b]$ for some $b\subseteq|N|$, we need to show that $N\vDash \sigma_\leq[b]$. By assumption we can build a $\mu$-directed system of union $M$ and have $b\leq M$. By transitivity of $\leq\ $, $b\leq N$ and the conclusion follows. Conversely, by directed system axioms, it suffices to build a $\mu$-directed system $\langle M_\alpha\in K_\lambda:\alpha\in I \rangle$ of union $M$ such that for all $\alpha\in I$, $M_\alpha\leq N$. Since $(N,R)\vDash\forall b\ \big(\sigma_\leq^R(b)\rightarrow\sigma_\leq(b)\big)$, we can require $M_\alpha\leq M$ instead of $M_\alpha\leq N$. Such construction is possible by L\"{o}wenheim-Skolem and coherence axioms.  
\end{enumerate}
\end{proof}

As an application of \ref{charmod2} and \ref{charsub2}, we generalize \ref{vasapp}:
\begin{corollary}
Let ${\bf K}$ be a $\mu$-AEC in $L$, $\lambda\defeq\ls$ and $M,N$ be $L$-structures. If either $M$ or $N$ is in $K$ and $M\preceq_{L_{\infty,{\lambda}^+}(\mu\cdot\mu)}N$, then $M\leq N$ (and both are in $K$). 
\end{corollary}
\begin{proof}
Same proof as in \ref{vasapp}: Since $M\preceq_{L_{\infty,\lambda^+}}N$, $M\subseteq_{L_{(\lambda+I_2(\lambda,{\bf K}))^+,\lambda^+}(\mu\cdot\mu)}N$. In particular $M\vDash \sigma_{\bf K}$ iff $N\vDash\sigma_{\bf K}$. By \ref{charmod2}, either $M,N$ is in $K$ implies both are in $K$. On the other hand, the assumption implies $M\subseteq_{\{\sigma_\leq\}}N$. By \ref{charsub2}(1), $M\leq N$. 
\end{proof}

We now state the $\mu$-AEC version of \ref{presvar}, which is a variation to {\cite[Theorem 3.2]{muaec}}. 
\begin{theorem}\mylabel{presvarmu}{Theorem \thetheorem}
Let ${\bf K}$ be a $\mu$-AEC in $L$ and with L\"{o}wenheim-Skolem number $\ls$. Define $\chi\defeq\ls+I_2(\ls,{\bf K})$. There exists a $(<\mu)$-ary expansion $L'\supseteq L$ of size $\chi$, an $L'$-theory $T$ and a set of $L'$-types $\Gamma$ of size $\chi$ such that 
\begin{enumerate}
\item $K=PC^\mu(T,\Gamma,L)$.
\item If $M',N'\in EC(T,\Gamma)$ and $M'\subseteq_{L'} N'$, then $M'\restriction L\lk N'\restriction L$.
\item If $M\lk N$, there are $L'$-expansions of $M,N$ to $M',N'$ such that $M'\subseteq_{L'}N'$. 
\end{enumerate}
\end{theorem}
\begin{proof}[Proof sketch]
Repeat the same argument in \ref{presvar} by replacing $\omega$ by $\mu$, in particular:
\begin{enumerate}
\item Superscripts of $f,g,h$ will be $\alpha,\beta,\gamma<\mu$ instead of $m,n,l<\omega$.
\item We require that the $K$-substructures generated by $\{f_k^\alpha:k<\ls,\alpha<\omega\}$ are $\mu$-directed instead of $\al$-directed.
\item The sentences $\{\sigma^\alpha:\alpha<\mu\}\cup\{\sigma^{\alpha,\beta,\gamma}:\alpha,\beta,\gamma<\mu\}$ are in $L'_{\chi^+,\mu}$.
\item For $i<I(\ls,{\bf K})$ and $j<I_2(\ls,{\bf K})$, the formulas $\phi_i,\psi_j$ are still $\ls$-conjunctions because ${\ls}^{<\mu}=\ls$. 
\item Chang's presentation theorem generalizes to $\mu$-AECs and converts a $L'_{\chi^+,\mu}$ theory of size $\chi$ into a $PC^\mu_\chi$.
\item When checking the items of the theorem statement, notice that by definition of a $\mu$-AEC, directed system axioms (instead of chain axioms) are built-in. Meanwhile, \ref{dirfact2} generalizes to $\mu$-directed systems.
\end{enumerate}
\end{proof}
\nocite{kue2} \nocite{kue0}
Unlike \ref{presvarcor}, the above result does not lead to the Hanf number computation because the languages are not finitary while well-ordering is definable. In particular there is no reasonable bound to the Hanf number of $L_{\aleph_1,\aleph_1}$ \cite[Chapter 5.1B]{dic}. As asked in \cite[Remark 3.3]{muaec}:
\begin{ques} Let $\mu\geq\aleph_1$. Does the Hanf number exist for $\mu$-AECs?\end{ques}
\bibliographystyle{alpha}
\bibliography{references}

{\small\setlength{\parindent}{0pt}
\textit{Email}: wangchil@andrew.cmu.edu

\textit{URL}: http://www.math.cmu.edu/$\sim$wangchil/

\textit{Address}: {Department of Mathematical Sciences, Carnegie Mellon University, Pittsburgh PA 15213, USA}
\end{document}